\documentclass[3p,times]{czysty}

\usepackage{czysty}

\volume{00}

\firstpage{1}

\runauth{}

\usepackage{amssymb}
\usepackage{amsthm}
\usepackage{amsmath}
\usepackage[figuresright]{rotating}
\usepackage[super,negative]{nth}
\usepackage{bbm}
\usepackage[g]{esvect}
\usepackage{tikz}
\usetikzlibrary{arrows}
\usetikzlibrary{automata}

\newtheorem{thm}{Theorem}[section]
\newtheorem{prop}[thm]{Proposition}
\newdefinition{df}[thm]{Definition}
\newdefinition{exa}[thm]{Example}
\newdefinition{rmk}[thm]{Remark}
\newproof{pf}{Proof}

\DeclareMathOperator{\md}{mod}
\DeclareMathOperator{\Dual}{D}
\DeclareMathOperator{\Hom}{Hom}
\DeclareMathOperator{\soc}{soc}
\DeclareMathOperator{\tp}{top}
\DeclareMathOperator{\Triv}{T}

\DeclareMathOperator{\rad}{rad}

\newcommand{\lo}{\left(}
\newcommand{\po}{\right)}
\newcommand{\lk}{\left\{}
\newcommand{\pk}{\right\}}
\newcommand{\lkw}{\left[}
\newcommand{\pkw}{\right]}
\newcommand{\grot}{stealth'}

\begin{document}

\begin{frontmatter}

\title{Weakly symmetric biserial algebras}

\author[]{Rafa\l{} Bocian\corref{mycorrespondingauthor}}
\ead{rafalb@mat.umk.pl}
\cortext[mycorrespondingauthor]{Corresponding author}
\author[]{Andrzej Skowro\'nski}
\ead{skowron@mat.uni.torun.pl}
\address{Faculty of Mathematics and Computer Science, Nicolaus Copernicus University, Chopina 12/18, 87-100 Toru\'n, Poland}

\begin{abstract}
We introduce the class of generalized biserial quiver algebras and prove that they provide a complete classification of all weakly symmetric biserial algebras over an algebraically closed field.
\end{abstract}


\end{frontmatter}

\vspace{-25eX}Dedicated to Ibrahim Assem on the occasion of his \nth{70} birthday\vspace{22eX}

\section{Introduction and main results}\label{roz1}

Throughout the paper, $K$ will denote a fixed algebraically closed field. By an algebra $A$ we mean an associative finite-dimensional $K$-algebra with identity, and we denote by $\md A$ the category of finite-dimensional right $A$-modules, and by $\Dual$ the standard duality $\Hom_{K}\lo -, K\po$ on $\md A$. An algebra $A$ is called \emph{self-injective} if $A_{A}$ is injective, or equivalently, the projective modules in $\md A$ are injective. Two self-injective algebras $A$ and $B$ are said to be \emph{socle equivalent} if the quotient algebras $A/\soc\lo A\po$ and $B/\soc\lo B\po$ are isomorphic. An important class of self-injective algebras is formed by \emph{symmetric algebras} $A$, for which there exists an associative, non-degenerate, symmetric $K$-bilinear form $\lo -,-\po:A\times A\longrightarrow K$. Classical examples of symmetric algebras include blocks of group algebras of finite groups and Hecke algebras of finite Coxeter groups. In fact, any algebra $A$ is a quotient algebra of its trivial extension algebra $\Triv\lo A\po=A\ltimes \Dual\lo A\po$, which is a symmetric algebra. The class of symmetric algebras is not closed under socle equivalences. Eighty years ago Nakayama and Nesbitt introduced in \cite{NN} the class of weakly symmetric algebras which contains symmetric algebras and is closed under socle equivalences. Recall that an algebra $A$ is \emph{weakly symmetric} if $\soc\lo P\po\cong\tp\lo P\po$, for any indecomposable projective module $P$ in $\md A$.

The main aim of the paper is to provide a complete classification of weakly symmetric biserial algebras. We recall that, following Fuller \cite{Fu}, an algebra $A$ is called \emph{biserial} if the radical of any non-uniserial indecomposable projective, left or right, $A$-module is a sum of two uniserial submodules whose intersection is simple or zero. A distinguished class of biserial algebras is formed by the special biserial algebras introduced in \cite{SW}, where it was shown that every representation-finite biserial algebra is special biserial. Moreover, it follows from \cite{PS} that a basic biserial algebra is special biserial if and only if $A$ admits a simply connected Galois covering. The special biserial algebras are tame and their representation theory is well understood (see \cite{BR,DS,WW}). Crawley-Boevey has proved in \cite{CB} that all biserial algebras are also tame, applying geometric degenerations of algebras and combinatorial results on the structure of biserial algebras established in his joint paper with Vila-Freyer \cite{CBVF}. However the representation theory of arbitrary biserial algebras is still an open problem. We note that there are many biserial algebras which are not special biserial.

The class of symmetric special biserial algebras coincides with the prominent class of Brauer graph algebras (see \cite{Ro,ES4,Sch}). These occurred in the classification (up to Morita equivalence) of representation-finite blocks of group algebras \cite{D,J,K}, symmetric algebras of Dynkin type $\mathbbm{A}_{n}$ \cite{GR,Rie}, symmetric algebras of Euclidean type $\widetilde{\mathbbm{A}}_{n}$ \cite{BS1}, the Gelfand-Ponomarev classification of singular Harish-Chandra modules over the Lorentz group \cite{GP}, restricted Lie algebras, or more generally infinitesimal group schemes \cite{FS1,FS2}, and in classifications of Hecke algebras \cite{AIP,AP,EN}. Symmetric biserial but not special biserial algebras occurred naturally in the classification of algebras of dihedral type \cite{E1,E2}, and more generally, algebras of generalized dihedral type \cite{ES3}, and all these algebras are over algebraically closed fields of characteristic $2$.

In our article \cite{BS2} (see also \cite{S}), we constructed weakly symmetric biserial algebras over algebraically closed field of arbitrary characteristic (using specific Brauer graphs with one loop), which are not special biserial, and proved that they provide a complete classification of all non-standard representation infinite self-injective algebras of domestic type. Here, we extend this construction to arbitrary Brauer graphs, using ideas of recent articles \cite{ES2,ES3,ES4}. Namely, Brauer graph algebras are interpreted in \cite{ES4} as biserial quiver algebras $B\lo Q, f, m_{\bullet}\po$ associated to weighted biserial quivers $\lo Q, f, m_{\bullet}\po$. An important class of Brauer graph algebras is formed by biserial weighted surface algebras $B\lo S, \vv{T}, m_{\bullet}\po=B\lo Q\lo S, \vv{T}\po, f, m_{\bullet}\po$, introduced under studied in \cite{ES2,ES3} in connection to tame blocks of group algebras. Here $\lo Q\lo S,\vv{T}\po,f\po$ is the triangulation quiver associated to a given triangulation $T$ of a $2$-dimensional real compact surface $S$ with or without boundary, and an orientation $\vv{T}$ of triangles in $T$. Moreover, the symmetric algebras socle equivalent to Brauer graph algebras are described in \cite{ES4} as biserial quiver algebras with border $B\lo Q,f,m_{\bullet},b_{\bullet}\po$.

The following theorem is the main result of the paper.
\begin{thm}\label{tw_1_1}
Let $A$ be a basic, indecomposable, finite-dimensional, self-injective algebra of dimension at least $2$ over an algebraically closed field $K$. Then the following statements are equivalent.
\begin{enumerate}[(i)]
\item $A$ is a weakly symmetric biserial algebra.
\item $A$ is isomorphic to a generalized biserial quiver algebra $B\lo Q,f,m_{\bullet},r_{\bullet},c_{\bullet},d_{\bullet}\po$.
\end{enumerate}
\end{thm}
The main ingredients for the proof of Theorem \ref{tw_1_1} are the results on the combinatorial structure of arbitrary biserial algebras established in \cite{CBVF} and weakly symmetric versions of results for the symmetric biserial quiver algebras proved in \cite{ES2,ES3,ES4}.

This paper is organized as follows. Section \ref{roz2} is devoted to presenting a characterization of weakly symmetric special biserial algebras. In Section \ref{roz3} we introduce the generalized biserial quiver algebras and describe their basic properties. In the final Section \ref{roz4} we prove Theorem \ref{tw_1_1} and present two illustrating examples.

For general background on the relevant representation theory we refer to the books \cite{ASS,E2,SS,SY1,SY2}.

\section{Special biserial algebras}\label{roz2}
	
A \emph{quiver} is a quadruple $Q = \lo Q_{0}, Q_{1}, s, t\po$ consisting of a finite set $Q_{0}$ of vertices, a finite set $Q_{1}$ of arrows, and two maps $s,t : Q_{1} \longrightarrow Q_{0}$ which associate to each arrow $\alpha \in Q_{1}$ its source $s(\alpha) \in Q_{0}$ and  its target $t(\alpha) \in Q_{0}$. We denote by $KQ$ the path algebra of $Q$ over $K$ whose underlying $K$-vector space has as its basis the set of all paths in $Q$ of length $\geqslant 0$, and by $R_{Q}$ the arrow ideal of $KQ$ generated by all paths in $Q$ of length $\geqslant 1$. An ideal $I$ in $K Q$ is said to be \emph{admissible} if there exists $m \geqslant 2$ such that $R_{Q}^{m} \subseteq I \subseteq R_{Q}^{2}$. If $I$ is an admissible ideal in $KQ$, then the quotient algebra $KQ/I$ is called a \emph{bound quiver algebra}, and is a finite-dimensional basic $K$-algebra. Moreover, $KQ/I$ is indecomposable if and only if $Q$ is connected. Every basic, indecomposable, finite-dimensional $K$-algebra $A$ has a bound quiver presentation $A \cong KQ/I$, where $Q = Q_{A}$ is the \emph{Gabriel quiver} of $A$ and $I$ is an admissible ideal in $KQ$. For a bound quiver algebra $A=KQ/I$, we denote by $e_{i}$, $i \in Q_{0}$, the associated complete set of pairwise orthogonal primitive idempotents of $A$. Then the modules $S_{i}=e_{i}A/e_{i}\rad A$ (respectively, $P_{i}=e_{i}A$), $i \in Q_{0}$, form a complete family of pairwise non-isomorphic simple modules (respectively, indecomposable projective modules) in $\md A$.

Following \cite{SW}, an algebra $A$ is said to be \emph{special biserial} if $A$ is isomorphic to a bound quiver algebra $KQ/I$, where the bound quiver $\lo Q,I\po$ satisfies the following conditions:
\begin{enumerate}[(a)]
\item each vertex of $Q$ is a source and target of at most two arrows;
\item for any arrow $\alpha$ in $Q$ there are at most one arrow $\beta$ and at most one arrow $\gamma$ with $\alpha \beta \notin I$ and $\gamma \alpha \notin I$.
\end{enumerate}

Following \cite{ES4}, a \emph{biserial quiver} is a pair $\lo Q,f\po$, where $Q=\lo Q_{0},Q_{1},s,t\po$ is a finite connected quiver and $f : Q_{1} \longrightarrow Q_{1}$ is a permutation of the set of arrows of $Q$ satisfying the following conditions:
\begin{enumerate}[(a)]
\item $Q$ is $2$-regular, that is every vertex of $Q$ is the source and target of exactly two arrows;
\item for each arrow $\alpha \in Q_{1}$ we have $s\lo f\lo\alpha\po\po = t\lo\alpha\po$.
\end{enumerate}

Let $\lo Q,f\po$ be a biserial quiver. Then there is a canonical involution $\overline{\phantom{a}}:Q_{1}\longrightarrow Q_{1}$ which assigns to an arrow $\alpha\in Q_{1}$ the other arrow $\overline{\alpha}$ starting at $s\lo\alpha\po$. Hence we obtain another permutation $g : Q_{1} \longrightarrow Q_{1}$ defined by $g\lo\alpha\po=\overline{f\lo\alpha\po}$ for any $\alpha \in Q_{1}$, so that $f\lo\alpha\po$ and
$g\lo\alpha\po$ are the arrows starting at $t\lo\alpha\po$. Let $\mathcal{O}\lo\alpha\po$ be the $g$-orbit of an arrow $\alpha$ and set $n_{\alpha}=n_{\mathcal{O}\lo\alpha\po}=|\mathcal{O}\lo\alpha\po|$. We note that for each arrow $\alpha\in Q_{1}$, the path $\alpha g\lo\alpha\po \dots g^{n_{\alpha}-1}\lo\alpha\po$ must be a cycle, because $g$, as a permutation, is in particular one-to-one function. We denote by $\mathcal{O}\lo g\po$ the set of all $g$-orbits in $Q_{1}$. A function
\begin{displaymath}
m_{\bullet} : \mathcal{O}\lo g\po \longrightarrow \mathbbm{N}^{*} = \mathbbm{N} \setminus \lk 0\pk
\end{displaymath}
is called a \emph{weight function} of $\lo Q,f\po$. We write briefly $m_{\alpha} = m_{\mathcal{O}\lo\alpha\po}$ for $\alpha \in Q_{1}$. For each arrow $\alpha \in Q_{1}$, we consider the oriented cycle
\begin{displaymath}
B_{\alpha}=\lo\alpha g\lo\alpha\po \dots g^{n_{\alpha}-1}\lo\alpha\po\po^{m_{\alpha}}
\end{displaymath}
of length $m_{\alpha}n_{\alpha}$. Moreover, a function
\begin{displaymath}
c_{\bullet} : Q_{1} \longrightarrow K^{*} = K \setminus \lk 0\pk
\end{displaymath}
is said to be a \emph{parameter function}.

\begin{df}\label{def_2_1}
Assume $\lo Q,f\po$ is a biserial quiver, $m_{\bullet}$ a weight function and $c_{\bullet}$ a parameter function of $\lo Q,f\po$. We define the quotient algebra
\begin{displaymath}
B\lo Q,f,m_{\bullet},c_{\bullet}\po=KQ/J\lo Q,f,m_{\bullet},c_{\bullet}\po,
\end{displaymath}
where $J\lo Q,f,m_{\bullet},c_{\bullet}\po$ is the ideal of the path algebra $KQ$ generated by the following elements:
\begin{enumerate}[(i)]
\item $\alpha f\lo\alpha\po$, for all arrows $\alpha\in Q_{1}$;
\item $c_{\alpha}B_{\alpha}-c_{\overline{\alpha}}B_{\overline{\alpha}}$, for all arrows $\alpha\in Q_{1}$.
\end{enumerate}
\end{df}
Then it follows from the proof of \cite[Proposition 2.3]{ES4} that $B\lo Q,f,m_{\bullet},c_{\bullet}\po$ is a weakly symmetric special biserial algebra of dimension $\sum_{\mathcal{O}\in\mathcal{O}\lo g\po}m_{\mathcal{O}} n_{\mathcal{O}}^{2}$. Hence, $B\lo Q,f,m_{\bullet},c_{\bullet}\po$ can be considered as a weakly symmetric biserial quiver algebra (a weakly symmetric version of a biserial quiver algebra introduced in \cite{ES4}).

Moreover, we have the following consequence of \cite[Theorem 2.6]{ES4} and the proof of Theorem \ref{tw_1_1}.
\begin{thm}
Let $A$ be a basic, indecomposable algebra of dimension at least $3$, over an algebraically closed field $K$. The following statements are equivalent:
\begin{enumerate}[(i)]
\item $A$ is a weakly symmetric special biserial algebra;
\item $A$ is isomorphic to a weakly symmetric biserial quiver algebra $B\lo Q,f,n_{\bullet},c_{\bullet}\po$.
\end{enumerate}
\end{thm}

\section{Generalized biserial quiver algebras}\label{roz3}

Let $\lo Q,f\po$ be a biserial quiver and $g$ the associated permutation of $Q_{1}$ with $g\lo\alpha\po=\overline{f\lo\alpha\po}$ for any arrow $\alpha \in Q_{1}$. We keep the notation introduced in Section \ref{roz2}. An arrow $\alpha\in Q_{1}$ is said to be \emph{an admissible arrow} of $\lo Q,f\po$ if $f^{2}\lo\alpha\po$ belongs to the $g$-orbit $\mathcal{O}\lo\alpha\po$ of $\alpha$. We denote by $\Omega\lo Q,f\po$ the set of all admissible arrows of $\lo Q,f\po$. For each arrow $\alpha\in\Omega\lo Q,f\po$, we consider the following subpath of $\alpha g\lo\alpha\po \dots g^{n_{\alpha}-1}\lo\alpha\po$
\begin{displaymath}
C_{\alpha}=\alpha g\lo\alpha\po \dots g^{-1}\lo f^{2}\lo\alpha\po\po.
\end{displaymath}

Let $m_{\bullet} : \mathcal{O}\lo g\po \longrightarrow \mathbbm{N}^{*}$ be a weight function of $\lo Q,f\po$. Then a function
\begin{displaymath}
r_{\bullet} : \Omega\lo Q,f\po \longrightarrow \mathbbm{N}^{*}
\end{displaymath}
is said to be a \emph{rank function} of the weighted biserial quiver $\lo Q,f, m_{\bullet}\po$ if $r_{\alpha}\leqslant m_{\alpha}$ for any arrow $\alpha\in\Omega\lo Q,f\po$. Then we may associate to any arrow $\alpha\in\Omega\lo Q,f\po$ the path
\begin{displaymath}
D_{\alpha}=\lo\alpha g\lo\alpha\po \dots g^{n_{\alpha}-1}\lo\alpha\po\po^{r_{\alpha}-1}C_{\alpha}
\end{displaymath}
from $s\lo\alpha\po$ to $t\lo f\lo\alpha\po\po=s\lo f^{2}\lo\alpha\po\po$. We note that $D_{\alpha}$ is a subpath of the cycle $B_{\alpha}=\lo\alpha g\lo\alpha\po \dots g^{n_{\alpha}-1}\lo\alpha\po\po^{m_{\alpha}}$. An \emph{admissible function} of $\lo Q,f,m_{\bullet}\po$ is a function
\begin{displaymath}
d_{\bullet} : \Omega\lo Q,f\po \longrightarrow K
\end{displaymath}
satisfying the following conditions:
\begin{enumerate}[(i)]
\item $d_{f^{-1}\lo\alpha\po}=0$ for any arrow $\alpha\in Q_{1}$ with $m_{\alpha}n_{\alpha}=1$;
\item \label{warunek2} $d_{\alpha}d_{\beta}\neq 1$ for any arrows $\alpha\neq\beta$ in $\Omega\lo Q,f\po$ with $s\lo g\lo\alpha\po\po=s\lo g\lo\beta\po\po$, $t\lo g\lo\alpha\po\po=t\lo g\lo\beta\po\po$, $d_{\alpha},d_{\beta}\in K^{*}$.
\end{enumerate}

We note that the condition (\ref{warunek2}) is strongly related with the condition ($C2$) in \cite[Corollary 3]{CBVF}.

\begin{df}\label{def_3_1}
Let $\lo Q,f\po$ be a biserial quiver, $m_{\bullet}$ a weight function, $r_{\bullet}$ a rank function, $c_{\bullet}$ a parameter function of $\lo Q,f\po$, and $d_{\bullet}$ an admissible function of $\lo Q,f,m_{\bullet}\po$. We define the quotient algebra
\begin{displaymath}
B\lo Q,f,m_{\bullet},r_{\bullet},c_{\bullet},d_{\bullet}\po=KQ/J\lo Q,f,m_{\bullet},r_{\bullet},c_{\bullet},d_{\bullet}\po,
\end{displaymath}
where $J\lo Q,f,m_{\bullet},r_{\bullet},c_{\bullet},d_{\bullet}\po$ is the ideal in the path algebra $KQ$ of $Q$ over $K$ defined by the elements:
\begin{enumerate}[(i)]
\item $\alpha f\lo\alpha\po$, for all arrows $\alpha\in Q_{1}\setminus\Omega\lo Q,f\po$;
\item $\alpha f\lo\alpha\po-d_{\alpha}D_{\alpha}$, for all arrows $\alpha\in\Omega\lo Q,f\po$;
\item $c_{\alpha}B_{\alpha}-c_{\overline{\alpha}}B_{\overline{\alpha}}$, for all arrows $\alpha\in Q_{1}$;
\item $B_{\alpha}\alpha$, for all arrows $\alpha\in Q_{1}$.
\end{enumerate}
\end{df}
Then $B\lo Q,f,m_{\bullet},r_{\bullet},c_{\bullet},d_{\bullet}\po$ is called a \emph{generalized biserial quiver algebra}. We note that if $d_{\bullet}$ is a zero function ($d_{\alpha}=0$ for any $\alpha\in\Omega\lo Q,f\po$) then $B\lo Q,f,m_{\bullet},r_{\bullet},c_{\bullet},d_{\bullet}\po$ is the weakly symmetric special biserial algebra $B\lo Q,f,m_{\bullet},c_{\bullet}\po$ defined in Section \ref{roz2}.
\begin{rmk}
Let $B=B\lo Q,f,m_{\bullet},r_{\bullet},c_{\bullet},r_{\bullet}\po$ be a generalized biserial quiver algebra of dimension at least $3$. Assume $\alpha\in Q_{1}$ is an arrow with $m_{\alpha}n_{\alpha}=1$. Then $\alpha$ is a loop, $B_{\overline{\alpha}}$ is a cycle of length at least $2$, and $c_{\alpha}\alpha=c_{\overline{\alpha}}B_{\overline{\alpha}}$. Hence $\alpha$ belongs to the square of the radical of $B$, and so is not arrow of the Gabriel quiver $Q_{B}$ of $B$, and we call it a \emph{virtual loop}.
\end{rmk}
The following proposition describes basic properties of generalized biserial quiver algebras.
\begin{prop}\label{stw_3_3}
Let $B=B\lo Q,f,m_{\bullet},r_{\bullet},c_{\bullet},d_{\bullet}\po$ be a generalized biserial quiver algebra. Then $B$ is a finite-dimensional, weakly symmetric, biserial algebra with $\dim_{K}{B}=\sum_{\mathcal{O}\in\mathcal{O}\lo g\po}m_{\mathcal{O}} n_{\mathcal{O}}^{2}$.
\end{prop}
\begin{proof}
We may assume that $B$ is of dimensional at least $3$. It follows from the definition that $B$ is of the form $KQ_{B}/I_{B}$, where $Q_{B}$ is obtained from $Q$ by remaining all virtual loops, and $I_{B}=J\lo Q,f,m_{\bullet},r_{\bullet},c_{\bullet},d_{\bullet}\po\cap KQ_{B}$. Let $i$ be a vertex of $Q$ and $\alpha$, $\overline{\alpha}$ the two arrows in $Q$ starting at $i$. If $\alpha\in\Omega\lo Q,f\po$ (respectively, $\overline{\alpha}\in\Omega\lo Q,f\po$) then $\alpha f\lo\alpha\po=d_{\alpha}D_{\alpha}$ (respectively, $\overline{\alpha}f\lo\overline{\alpha}\po=d_{\overline{\alpha}}D_{\overline{\alpha}}$) belongs to the uniserial right module $\alpha B$ (respectively, $\overline{\alpha}B$). Then the indecomposable projective $B$-module $P_{i}=e_{i}B$ has a basis given by $e_{i}$ together with all initial proper subwords of $B_{\alpha}$ and $B_{\overline{\alpha}}$, $c_{\alpha}B_{\alpha}=c_{\overline{\alpha}}B_{\overline{\alpha}}$, and hence $\dim_{K}{P_{i}}=m_{\alpha}n_{\alpha}+m_{\overline{\alpha}}n_{\overline{\alpha}}$. We also note that $B_{\alpha}$ (respectively, $B_{\overline{\alpha}}$) generates the socle of $P_{i}$, which is a simple module isomorphic to $S_{i}=\tp{P_{i}}$. Hence $B$ is a weakly symmetric biserial algebra. Clearly, the union of the chosen bases of the indecomposable projective $B$-modules gives a basis of $B$, and we deduce that $\dim_{K}{B} = \sum_{\mathcal{O}\in\mathcal{O}\lo g\po}m_{\mathcal{O}} n_{\mathcal{O}}^{2}$.
\end{proof}
A generalized biserial quiver algebra $B\lo Q,f,m_{\bullet},r_{\bullet},c_{\bullet},d_{\bullet}\po$ with $d_{\bullet}$ a zero function is a special biserial algebra, not depending on the rank function $r_{\bullet}$, and is defined by $B\lo Q,f,m_{\bullet},c_{\bullet}\po$.
\begin{rmk}
Let $\lo Q,f\po$ be a biserial quiver. Following \cite{ES3}, a vertex $i\in Q_{0}$ is called a border vertex of $\lo Q,f\po$ if there is a loop $\alpha$ at $i$ with $f\lo\alpha\po=\alpha$. The set $\partial\lo Q,f\po$ of all border vertices of $\lo Q,f\po$ is called the border of $\lo Q,f\po$. Moreover, a loop $\alpha\in Q_{1}$ with $f\lo\alpha\po=\alpha$ is called a border loop of $\lo Q,f\po$. We note that every border loop $\alpha$ of $\lo Q,f\po$ is an admissible arrow of $\lo Q,f\po$, and hence the set $\Omega\lo Q,f\po$ contains the set $\partial\lo Q,f\po^{*}$ of all border loops of $\lo Q,f\po$. Assume that the border $\partial\lo Q,f\po$  is not empty. Then a function
\begin{displaymath}
b_{\bullet} : \partial\lo Q,f\po \longrightarrow K
\end{displaymath}
is called in \cite{ES3} a border function of $\lo Q,f\po$. To a border function $b_{\bullet}$ of $\lo Q,f\po$ we may associate an admissible function
\begin{displaymath}
b_{\bullet}^{*} : \Omega\lo Q,f\po \longrightarrow K
\end{displaymath}
such that $b_{\alpha}^{*}=b_{s\lo\alpha\po}$, for any loop $\alpha\in\partial\lo Q,f\po^{*}$, and $b_{\alpha}^{*}=0$ for all arrows $\alpha\in\Omega\lo Q,f\po\setminus\partial\lo Q,f\po^{*}$. Let $m_{\bullet}: \mathcal{O}\lo g\po \longrightarrow \mathbbm{N}^{*}$ be a weight function of $\lo Q,f\po$, and $r_{\bullet}^{*} : \Omega\lo Q,f\po \longrightarrow \mathbbm{N}^{*}$ the rank function with $r_{\alpha}^{*}=m_{\alpha}$ for any arrow $\alpha\in\Omega\lo Q,f\po$. Finally, let $c_{\bullet}: Q_{1}\longrightarrow K^{*}$ be the constant parameter function taking only the value $1$. Then the associated generalized biserial quiver algebra $B\lo Q,f,m_{\bullet},r_{\bullet}^{*},c_{\bullet},b_{\bullet}^{*}\po$ is the symmetric biserial algebra $B\lo Q,f,m_{\bullet},b_{\bullet}\po$ with border considered in \cite{ES3,ES4}.
\end{rmk}

We describe now local generalized biserial quiver algebras. We have two possibilities.
\begin{exa}\label{przy3_5}
Let $\lo Q,f\po$ be the biserial quiver with
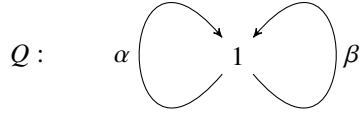
\begin{center}
\begin{tikzpicture}[auto=right,>=\grot]
\node (0) at (0,0) {$1$};
\path (0) edge [->,in=50,out=310,loop,looseness=12] node {$\beta$} (0);
\path (0) edge [->,in=130,out=230,loop,looseness=12] node[left] {$\alpha$} (0);
\node[text width=0cm,text centered] (n1) at (-3,0) {$Q:$};
\end{tikzpicture}
\end{center}
and $f\lo\alpha\po=\alpha$, $f\lo\beta\po=\beta$. Then $g\lo\alpha\po=\beta$, $g\lo\beta\po=\alpha$, and we have only one $g$-orbit $\mathcal{O}\lo\alpha\po=\lo\alpha\ \beta\po=\mathcal{O}\lo\beta\po$. Hence a multiplicity function $m_{\bullet}:\mathcal{O}\lo g\po\longrightarrow\mathbbm{N}^{*}$ is given by a positive integer $m=m_{\alpha}=m_{\beta}$. Moreover, $f^{2}\lo\alpha\po\in\mathcal{O}\lo\alpha\po$, $f^{2}\lo\beta\po\in\mathcal{O}\lo\beta\po$, and hence $\Omega\lo Q,f\po=Q_{1}$. Let $c_{\bullet}:Q_{1}\longrightarrow K^{*}$ be a parameter function, $r_{\bullet} : \Omega\lo Q,f\po \longrightarrow \mathbbm{N}^{*}$ a rank function, and $d_{\bullet} : \Omega\lo Q,f\po \longrightarrow K$ an admissible function of $\lo Q,f,m_{\bullet}\po$. Then the associated generalized biserial quiver algebra $B=B\lo Q,f,m_{\bullet},r_{\bullet},c_{\bullet},d_{\bullet}\po$ is given by the quiver $Q$ and the relations
\begin{alignat*}{5}
\alpha^{2}&=d_{\alpha}\lo\alpha\beta\po^{r_{\alpha}},\qquad & \beta^{2}&=d_{\beta}\lo\beta\alpha\po^{r_{\beta}},\qquad & c_{\alpha}\lo\alpha\beta\po^{m}&=c_{\beta}\lo\beta\alpha\po^{m},\qquad & \lo\alpha\beta\po^{m}\alpha&=0,\qquad & \lo\beta\alpha\po^{m}\beta&=0.
\end{alignat*}
\end{exa}
\begin{exa}
Let $\lo Q,f\po$ be the biserial quiver with $Q$ as in Example \ref{przy3_5} and $f\lo\alpha\po=\beta$, $f\lo\beta\po=\alpha$. Then $g\lo\alpha\po=\alpha$, $g\lo\beta\po=\beta$, and we have two $g$-orbits $\mathcal{O}\lo\alpha\po=\lo\alpha\po$, $\mathcal{O}\lo\beta\po=\lo\beta\po$. Since $f^{2}\lo\alpha\po=\alpha\in\mathcal{O}\lo\alpha\po$ and $f^{2}\lo\beta\po=\beta\in\mathcal{O}\lo\beta\po$, we have $\Omega\lo Q,f\po=Q_{1}$. Further, a multiplicity function $m_{\bullet}:\mathcal{O}\lo g\po\longrightarrow\mathbbm{N}^{*}$ is given by two positive integers $m=m_{\alpha}$, and $n=m_{\beta}$. Let $c_{\bullet}:Q_{1}\longrightarrow K^{*}$ be a parameter function, $r_{\bullet} : \Omega\lo Q,f\po \longrightarrow \mathbbm{N}^{*}$ a rank function, and $d_{\bullet} : \Omega\lo Q,f\po \longrightarrow K$ an admissible function of $\lo Q,f,m_{\bullet}\po$. Then the associated generalized biserial quiver algebra $B=B\lo Q,f,m_{\bullet},r_{\bullet},c_{\bullet},d_{\bullet}\po$ is given by the quiver $Q$ and the relations
\begin{alignat*}{5}
\alpha\beta&=d_{\alpha}\alpha^{r_{\alpha}},\qquad & \beta\alpha&=d_{\beta}\beta^{r_{\beta}},\qquad & c_{\alpha}\alpha^{m}&=c_{\beta}\beta^{n},\qquad & \alpha^{m+1}&=0,\qquad & \beta^{n+1}&=0.
\end{alignat*}
We note that for $m=1$, B is isomorphic to the truncated polynomial algebra $K\lkw X\pkw/\lo X^{n+1}\po$.
\end{exa}

\section{Proof of Theorem \ref{tw_1_1}}\label{roz4}

Let $Q=\lo Q_{0},Q_{1},s,t\po$ be a quiver. Following \cite{CBVF}, a \emph{bisection} of $Q$ is a pair $\lo \sigma,\tau\po$ of functions $\sigma, \tau : Q_{1}\longrightarrow \lk\pm 1\pk$ such that if $\alpha$ and $\beta$ are distinct arrows with $s\lo\alpha\po=s\lo\beta\po$ (respectively $t\lo\alpha\po=t\lo\beta\po$), then $\sigma\lo\alpha\po\neq\sigma\lo\beta\po$ (respectively $\tau\lo\alpha\po\neq\tau\lo\beta\po$). We note that a quiver $Q$ has a bisection if and only if the number of arrows in $Q$ with a prescribed source or sink is at most two. A \emph{bisected quiver} is a triple $\lo Q,\sigma,\tau\po$ consisting of a quiver $Q$ and its bisection $\lo\sigma,\tau\po$. Let $\lo Q,\sigma,\tau\po$ be a bisected quiver. A path $\alpha_{1}\alpha_{2}\ldots\alpha_{r}$ of $Q$ is said to be a $\lo \sigma,\tau\po$-{\it good path} (shortly, good path), if $\tau\lo\alpha_{i-1}\po=\sigma\lo\alpha_{i}\po$ for all $i\in\lk 2,3,\dots,r\pk$, otherwise it is $\lo \sigma,\tau\po$-{\it bad path} (shortly, bad path). If $i$ is a vertex of $Q$, then $e_{i}$ is a $\lo\sigma,\tau\po$-good path of length $0$. We denote by $\Lambda\lo Q,\sigma,\tau\po$ the set of all bad paths of length $2$ of a bisected quiver $\lo Q,\sigma,\tau\po$. If $\alpha$ is an arrow of $Q$ such that $|s^{-1}\lo t\lo\alpha\po\po|=2$ (that is, the vertex $t\lo\alpha\po$ is a source of two distinct arrows) then there is a unique arrow $\gamma$ such that $\alpha\gamma\in\Lambda\lo Q,\sigma,\tau\po$.
\begin{proof}[Proof of the Theorem \ref{tw_1_1}]
The implication (ii) $\Longrightarrow$ (i) follows from Proposition \ref{stw_3_3}. We prove (i) $\Longrightarrow$ (ii). Assume $A$ is a weakly symmetric biserial algebra. It follows from \cite[Corrolary 3]{CBVF} that there are a bisected quiver $\lo Q_{A},\sigma,\tau\po$ and a function $h_{\bullet}:\Lambda\lo Q_{A},\sigma,\tau\po\longrightarrow KQ_{A}$ satisfying the following conditions:
\begin{enumerate}[(C1)]
\item \label{biserial_pierwsza} $h_{\alpha\gamma}$ is either zero, or of the form $d_{\alpha\gamma}^{*}\delta_{1}\delta_{2}\dots\delta_{r}$ with $d_{\alpha\gamma}^{*}\in K^{*}$, $r\geqslant 1$ and $\alpha\delta_{1}\delta_{2}\dots\delta_{r}$ a good path with $t\lo\delta_{r}\po=t\lo\gamma\po$ and $\delta_{r}\neq\gamma$;
\item if $h_{\alpha\gamma}=d_{\alpha\gamma}^{*}\delta$ and $h_{\beta\delta}=d_{\beta\delta}^{*}\gamma$ with $d_{\alpha\gamma}^{*},d_{\beta\delta}^{*}\in K^{*}$, then $d_{\alpha\gamma}^{*}d_{\beta\delta}^{*}\neq 1$,
\end{enumerate}
and an admissible ideal $I$ in $KQ_{A}$ containing all elements $\alpha\gamma-\alpha h_{\alpha\gamma}$ for all $\alpha\gamma\in\Lambda\lo Q_{A},\sigma,\tau\po$, such that the algebra $A$ is isomorphic to $KQ_{A}/I$. Without lost of generality we may assume that $A= KQ_{A}/I$. We will define now a biserial quiver $Q$. Since $A$ is biserial, for each vertex $i\in Q_{A}$, we have $|s^{-1}\lo i\po| \leqslant 2$ and $|t^{-1}\lo i\po| \leqslant 2$. We claim that, for each vertex $i \in Q_{0}$, the equality $|s^{-1}\lo i\po|= |t^{-1}\lo i\po|$ holds. Indeed, let $|s^{-1}(i)| = 1$ for a vertex $i\in Q_{0}$. Then, by the biserial relations, the projective module $P_{i}=e_{i}A$ is uniserial and isomorphic to the injective hull of the simple module $S_{i}=\tp\lo P_{i}\po$, because $A$ is weakly symmetric. Hence $|t^{-1}\lo i\po|=1$. If $|t^{-1}\lo i\po|=1$ then applying similar arguments to the uniserial left projective module $Ae_{i}$, we conclude that $|s^{-1}\lo i\po| = 1$. We note that the opposite algebra $A^{op}$ of $A$ is also weakly symmetric and special biserial.

Let $\Delta$ be the set of all vertices $i\in\lo Q_{A}\po_{0}$ with $|s^{-1}\lo i\po|= 1 \}$ (equivalently, $|t^{-1}\lo i\po|= 1 \}$). We define $Q=\lo Q_{0},Q_{1},s,t\po$ with $Q_{0} = \lo Q_{A}\po_{0}$, $Q_{1}=\lo Q_{A}\po_{1}\bigcup\lk\eta_{i}: i\in\Delta\pk$, and $s\lo\eta_{i}\po=i=t\lo\eta_{i}\po$ for any $i\in\Delta$. Observe that $Q$ is a $2$-regular quiver. We define a permutation $f$ of $Q_{1}$. For each $i\in\Delta$, there are unique arrows $\alpha_{i}$ and $\beta_{i}$ in $\lo Q_{A}\po_{1}$ such that $s\lo\alpha_{i}\po=i=t\lo\beta_{i}\po$, and we set $f\lo\eta_{i}\po=\alpha_{i}$ and $f\lo\beta_{i}\po=\eta_{i}$. If $\alpha$ is an arrow of $Q_{A}$ with $t\lo\alpha\po\notin\Delta$, then we define $f\lo\alpha\po$ as the unique arrow in $\lo Q_{A}\po_{1}$ with $\alpha f\lo\alpha\po\in\Lambda\lo Q,\sigma,\tau\po$. Then $\lo Q, f\po$ is a biserial quiver.

We define now a weight function $m_{\bullet} : \mathcal{O}\lo g\po\longrightarrow\mathbbm{N}^{*}$, where $g=\overline{f}$. For each $i\in\Delta$, we have $g\lo\eta_{i}\po=\eta_{i}$, and we set $m_{\mathcal{O}\lo\eta_{i}\po}=1$. Let $\alpha$ be an arrow of $Q_{A}$, starting at vertex $j$, and let $n_{\alpha}=|\mathcal{O}\lo\alpha\po|$. Since $A$ is weakly symmetric biserial, there exists $m_{\alpha}\in\mathbbm{N}^{*}$ such  that
\begin{displaymath}
B_{\alpha}=\lo\alpha g\lo\alpha\po\dots g^{n_{\alpha}-1}\lo\alpha\po\po^{m_{\alpha}}
\end{displaymath}
is a maximal path in $Q_{A}$ which does not belong to $I$ and spans the socle $e_{j}A$. We claim that $m_{\alpha}=m_{g\lo\alpha\po}$. Indeed, let $k=t\lo\alpha\po$. Then $\lo g\lo\alpha\po \dots g^{n_{\alpha}-1}\lo\alpha\po\alpha\po^{m_{\alpha}-1} g\lo\alpha\po \dots g^{n_{\alpha}-1}\lo\alpha\po$ is a subpath of $B_{\alpha}$, and hence is not in $I$. We know also
\begin{displaymath}
B_{g\lo\alpha\po}=\lo g\lo\alpha\po \dots g^{n_{\alpha}-1}\lo\alpha\po\alpha\po^{m_{g\lo\alpha\po}}
\end{displaymath}
is a maximal path in $Q_{A}$ which does not belong to $I$ and spans the socle $e_{k}A$. Then
\begin{displaymath}
\lo g\lo\alpha\po \dots g^{n_{\alpha}-1}\lo\alpha\po\alpha\po^{m_{\alpha}}=\lkw\lo g\lo\alpha\po \dots g^{n_{\alpha}-1}\lo\alpha\po\alpha\po^{m_{\alpha-}-1}g\lo\alpha\po\dots g^{n_{\alpha}-1}\lo\alpha\po\pkw\alpha
\end{displaymath}
does not belong to $I$, and hence $m_{\alpha}\leqslant m_{g\lo\alpha\po}$. Now, repeating the above arguments, we obtain the inequalities $m_{\alpha}\leqslant m_{g\lo\alpha\po}\leqslant\dots\leqslant m_{g^{n_{\alpha}-1}\lo\alpha\po}\leqslant m_{\alpha}$, and consequently $m_{\bullet}$ is constant on the $g$-orbit $\mathcal{O}\lo\alpha\po$ of $\alpha$. Therefore, we may define a parameter function $m_{\bullet} : \mathcal{O}\lo g\po\longrightarrow\mathbbm{N}^{*}$ such that $m_{\mathcal{O}\lo\alpha\po}=m_{\alpha}$ for any arrow $\alpha\in \lo Q_{A}\po_{1}$ and $m_{\mathcal{O}\lo\eta_{i}\po}=1$ for any $i\in\Delta$.

We define now a parameter function $c_{\bullet} : Q_{1}\longrightarrow K^{*}$. Let $i\in\Delta$ and $\alpha_{i}$ be the unique arrow in $\lo Q_{A}\po_{1}$ starting at $i$. Then we set $c_{\eta_{i}}=1=c_{\alpha_{i}}$. Assume now that $j$ is a vertex in $\lo Q_{A}\po_{0}\setminus\Delta$, and $\alpha$, $\overline{\alpha}$ the two arrows starting at $j$. Then each of the paths $B_{\alpha}$ and $B_{\overline{\alpha}}$ spans the socle of $e_{j}A$, and hence there are $c_{\alpha}$ and $c_{\overline{\alpha}}$ in $K^{*}$ such that $c_{\alpha}B_{\alpha}=c_{\overline{\alpha}}B_{\overline{\alpha}}$.

We define now a rank function $r_{\bullet} : \Omega\lo Q,f\po\longrightarrow\mathbbm{N}^{*}$ and an admissible function $d_{\bullet} : \Omega\lo Q,f\po\longrightarrow K$. For each $i\in\Delta$, $\eta_{i}\in\Omega\lo Q,f\po$ and there is the unique $\beta_{i}$ in $\lo Q_{A}\po$ such that $t\lo\beta_{i}\po=i$, and we set $d_{\beta_{i}}=0$. Let $\alpha\in\Omega\lo Q,f\po$ be such that $t\lo\alpha\po\notin\Delta$. From definition of the permutation $f$, we have $\alpha f\lo\alpha\po\in\Lambda\lo Q,\sigma,\tau\po$. If $h_{\alpha f\lo\alpha\po}=0$ then we set $r_{\alpha}=1$ and $d_{\alpha}=0$. Assume $h_{\alpha f\lo\alpha\po}\neq 0$. Then there are $d_{\alpha f\lo\alpha\po}^{*}\in K^{*}$ and arrows $\delta_{1}$, $\delta_{2}$, $\dots$, $\delta_{r}$ of $Q_{A}$, with $r\geqslant 1$ and $\alpha\delta_{1}\delta_{2}\dots\delta_{r}$ a good path with $t\lo\delta_{r}\po=t\lo f\lo\alpha\po\po$ and $\delta_{r}\neq f\lo\alpha\po$, such that $h_{\alpha f\lo\alpha\po}=d_{\alpha f\lo\alpha\po}^{*}\delta_{1}\delta_{2}\dots\delta_{r}$. We set $d_{\alpha}=d_{\alpha f\lo\alpha\po}^{*}$. From definition of the permutation $g$, we have $\delta_{1}$, $\delta_{2}$, $\dots$, $\delta_{r}\in\mathcal{O}\lo\alpha\po$ and hence $\alpha\delta_{1}\delta_{2}\dots\delta_{r}$ is a subpath of $B_{\alpha}$. Moreover, we have the equality $g^{-1}\lo f^{2}\lo\alpha\po\po=\delta_
{r}$. Then there is $t_{\alpha}\in\mathbbm{N}$ such that $\alpha\delta_{1}\delta_{2}\dots\delta_{r}=\lo\alpha g\lo\alpha\po g^{2}\lo\alpha\po\dots g^{n_{\alpha}-1}\lo\alpha\po\po^{t_{\alpha}}\alpha g\lo\alpha\po\dots g^{-1}\lo f^{2}\lo\alpha\po\po$. We set $r_{\alpha}=t_{\alpha}+1$. With these biserial quiver $\lo Q,f\po$, weight function $m_{\bullet}$, parameter function $c_{\bullet}$, rank function $r_{\bullet}$ and admissible function $d_{\bullet}$ there is a canonical isomorphism of $K$-algebras $A=KQ_{A}/I\longrightarrow B\lo Q,f,m_{\bullet},r_{\bullet},c_{\bullet},d_{\bullet}\po$.
\end{proof}
\begin{exa}\label{przy4_1}
Let $\lo Q,f\po$ be the biserial quiver
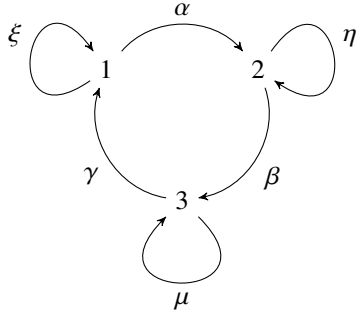
\begin{center}
\begin{tikzpicture}[auto=left,>=\grot]
\node (drugi) at (1,1) {$2$};
\node (pierwszy) at (-1,1){$1$};
\node (trzeci) at (0,-0.74) {$3$};
\path (pierwszy) edge [->,bend right=-50] node {$\alpha$} (drugi);
\path (drugi) edge [->,bend right=-50] node {$\beta$} (trzeci);
\path (trzeci) edge [->,bend right=-50] node {$\gamma$} (pierwszy);
\path (pierwszy) edge [->,in=125,out=215,loop,looseness=10] node {$\xi$} (pierwszy);
\path (drugi) edge [->,in=325,out=55,loop,looseness=10] node {$\eta$} (drugi);
\path (trzeci) edge [->,in=225,out=315,loop,looseness=10] node {$\mu$} (trzeci);
\end{tikzpicture}
\end{center}
with $f$-orbits $\lo\alpha\ \beta\ \gamma\po$, $\lo\xi\po$, $\lo\eta\po$, $\lo\mu\po$. Then $\mathcal{O}\lo g\po$ consists of one $g$-orbit $\mathcal{O}\lo\alpha\po=\lo\alpha\ \eta\ \beta\ \mu\ \gamma\ \xi\po$. Hence we have $\Omega\lo Q,f\po=Q_{1}$. Let $m_{\bullet}:\mathcal{O}\lo g\po\longrightarrow\mathbbm{N}^{*}$ be the weight function with $m_{\mathcal{O}\lo\alpha\po}=2$ and $r_{\bullet}: Q_{1}\longrightarrow\mathbbm{N}^{*}$ the rank function such that $r_{\xi}=r_{\alpha}=r_{\gamma}=1$ and $r_{\eta}=r_{\beta}=r_{\mu}=2$. We take the parameter function $c_{\bullet}: Q_{1}\longrightarrow K^{*}$ with $c_{\xi}=c_{\eta}=c_{\mu}=c_{\gamma}=1$, $c_{\alpha}=c_{\beta}=-1$. Moreover, let $d_{\bullet}: Q_{1}\longrightarrow K$ be the admissible function such that $d_{\xi}=d_{\beta}=1$, $d_{\eta}=d_{\mu}=0$, $d_{\alpha}=d_{\gamma}=-1$. Then the generalized biserial quiver algebra $B=B\lo Q,f,m_{\bullet},r_{\bullet},c_{\bullet},d_{\bullet}\po$ is given by the quiver $Q$ and the relations
\begin{alignat*}{5}
\alpha\beta&=-\alpha\eta\beta\mu, & \xi^{2}&=\xi\alpha\eta\beta\mu\gamma,\qquad & \lo\xi\alpha\eta\beta\mu\gamma\po^{2}&=-\lo\alpha\eta\beta\mu\gamma\xi\po^{2},\qquad & \lo\alpha\eta\beta\mu\gamma\xi\po^{2}\alpha &=0,\qquad & \lo\xi\alpha\eta\beta\mu\gamma\po^{2}\xi &=0, \\
\beta\gamma&=\beta\mu\gamma\xi\alpha\eta\beta\mu\gamma\xi,\qquad & \eta^{2}&=0, & \lo\eta\beta\mu\gamma\xi\alpha\po^{2}&=-\lo\beta\mu\gamma\xi\alpha\eta\po^{2}, & \lo\beta\mu\gamma\xi\alpha\eta\po^{2}\beta &=0, & \lo\eta\beta\mu\gamma\xi\alpha\po^{2}\eta&=0, \\ 
\gamma\alpha&=-\gamma\xi\alpha\eta, & \mu^{2}&=0, & \lo\mu\gamma\xi\alpha\eta\beta\po^{2}&=\lo\gamma\xi\alpha\eta\beta\mu\po^{2}, & \lo\gamma\xi\alpha\eta\beta\mu\po^{2}\gamma &=0, & \lo\mu\gamma\xi\alpha\eta\beta\po^{2}\mu&=0.
\end{alignat*}
It follows from Proposition \ref{stw_3_3} that $\dim_{K}B=2\cdot 6^{2}=72$. Further, if $K$ is of characteristic different from $2$, then $B$ is weakly symmetric but not symmetric, because we have the equality $\lo\xi\alpha\eta\beta\mu\gamma\po^{2}=-\lo\alpha\eta\beta\mu\gamma\xi\po^{2}$ of socle element of $B$.
\end{exa}
\begin{exa}
Let $\lo Q,f\po$ be the biserial quiver with $Q$ as in Example \ref{przy4_1} and $f$-orbits $\lo\alpha\ \beta\ \mu\ \gamma\ \xi\po$, $\lo\eta\po$. Then $\mathcal{O}\lo g\po$ consists of $g$-orbits $\mathcal{O}\lo\alpha\po=\lo\alpha\ \eta\ \beta\ \gamma\po$, $\mathcal{O}\lo\mu\po=\lo\mu\po$ and $\mathcal{O}\lo\xi\po=\lo\xi\po$. Hence we have $\Omega\lo Q,f\po=\lk\beta,\gamma,\eta\pk$. Let $m_{\bullet}:\mathcal{O}\lo g\po\longrightarrow\mathbbm{N}^{*}$ be the weight function with $m_{\mathcal{O}\lo\alpha\po}=3$, $m_{\mathcal{O}\lo\mu\po}=2$, $m_{\mathcal{O}\lo\xi\po}=1$ and $r_{\bullet}: \lk\beta,\gamma,\eta\pk\longrightarrow\mathbbm{N}^{*}$ the rank function such that $r_{\gamma}=1$, $r_{\beta}=2$ and $r_{\eta}=3$. We take the parameter function $c_{\bullet}: Q_{1}\longrightarrow K^{*}$ with $c_{\xi}=c_{\eta}=c_{\mu}=c_{\gamma}=1$, $c_{\alpha}=c_{\beta}=-1$. Moreover, let $d_{\bullet}: \lk\beta,\gamma,\eta\pk\longrightarrow K$ be the admissible function such that $d_{\beta}=1$, $d_{\eta}=7$, $d_{\gamma}=0$ (since $m_{\xi}n_{\xi}=1$ and $f^{-1}\lo\xi\po=\gamma$). Then the generalized biserial quiver algebra $B=B\lo Q,f,m_{\bullet},r_{\bullet},c_{\bullet},d_{\bullet}\po$ is given by the quiver $Q$ and the relations
\begin{alignat*}{5}
\alpha\beta&=0,\qquad & \beta\mu &=\beta\gamma\alpha\eta\beta, & \xi &=-\lo\alpha\eta\beta\gamma\po^{3}, & \lo\alpha\eta\beta\gamma\po^{3}\alpha &=0,\qquad & \xi^{2} &=0, \\
\mu\gamma&=0, & \eta^{2} &=7\lo\eta\beta\gamma\alpha\po^{2}\beta\gamma\alpha,\qquad & \lo\eta\beta\gamma\alpha\po^{3}&=-\lo\beta\gamma\alpha\eta\po^{3},\qquad & \lo\beta\gamma\alpha\eta\po^{3}\beta &=0, & \lo\eta\beta\gamma\alpha\po^{3}\eta&=0, \\
\xi\alpha&=0, & \gamma\xi&=0, & \mu^{2}&=\lo\gamma\alpha\eta\beta\po^{3}, & \lo\gamma\alpha\eta\beta\po^{3}\gamma &=0, & \mu^{3}&=0.
\end{alignat*}
It follows from Proposition \ref{stw_3_3} that $\dim_{K}B=3\cdot 4^{2}+2\cdot 1+1\cdot 1=51$.
\end{exa}

\section*{Acknowledgements}
The research was supported by the research grant DEC-2011/02/A/ST1/00216 of the National Science Centre Poland.

\bibliographystyle{elsarticle-num}

\begin{thebibliography}{00}
\bibitem[AIP]{AIP}
S.~Ariki, K.~Iijima, E.~Park, 
Representation type of finite quiver Hecke algebras of type $A^{(1)}_{\ell}$ for arbitrary parameters, 
Int. Math. Res. Not. IMRN 15 (2015), 6070--6135.

\bibitem[AP]{AP}
S.~Ariki, E.~Park,
Representation type of finite quiver Hecke algebras of type $D^{(2)}_{\ell+1}$,
Trans. Amer. Math. Soc. 368 (2016), 3211--3242.

\bibitem[ASS]{ASS}
I.~Assem, D.~Simson, A.~Skowro\'nski,
Elements of the Representation Theory of Associative Algebras 1: Techniques of Representation Theory,
London Math. Soc. Student Texts 65, Cambridge University Press, Cambridge 2006.

\bibitem[BS1]{BS1}
R.~Bocian, A.~Skowro\'nski,
Symmetric special biserial algebras of Euclidean type,
Colloq. Math. 96 (2003), 121--148.

\bibitem[BS2]{BS2}
R.~Bocian, A.~Skowro\'nski,
Socle deformations of selfinjective algebras of Euclidean type,
Comm. Algebra 34 (2006), 4235--4257.

\bibitem[BR]{BR}
M.~C.~R.~Buttler, C.~M.~Ringel,
Auslander-Reiten sequences with few middle terms and applications to string algebras,
Comm. Algebra 15 (1987), 145--179.

\bibitem[CB]{CB}
W.~Crawley-Boevey,
Tameness of biserial algebras,
Arch. Math. 65 (1995), 399--407.

\bibitem[D]{D}
E.~C.~Dade, 
Blocks with cyclic defect groups,
Ann. of Math. 84 1966, 20--48.

\bibitem[DS]{DS}
P.~Dowbor, A.~Skowro\'nski,
Galois coverings of representation-infinite algebras,
Comment. Math. Helv. 62 (1987), 311--337.

\bibitem[E1]{E1}
K.~Erdmann,
Algebras and dihedral defect group,
Proc. London Math. Soc. 54 (1987) 88--114.

\bibitem[E2]{E2}
K.~Erdmann,
Blocks of Tame Representation Type and Related Algebras,
Lecture Notes in Mathematics 1428, Springer-Verlag, 1990.

\bibitem[EN]{EN}
K.~Erdmann, D.~Nakano,
Representation type of Hecke algebras of type $A$,
Trans. Amer. Math. Soc. 354 (2002), 275--285.

\bibitem[ES1]{ES2}
K.~Erdmann, A.~Skowro\'nski,
Weighted surface algebras,
Preprint 2017, http://arxiv.org/abs/1703.02346.

\bibitem[ES2]{ES3}
K.~Erdmann, A.~Skowro\'nski,
Algebras of generalized dihedral type,
Preprint 2017, http://arxiv.org/abs/1706.00688.

\bibitem[ES3]{ES4}
K.~Erdmann, A.~Skowro\'nski,
From Brauer graph algebras to biserial weighted surface algebras, 
Preprint 2018, http://arxiv.org/abs/1706.07693v2.

\bibitem[FS1]{FS1}
R.~Farnsteiner, A.~Skowro\'nski,
Classification of restricted Lie algebras with tame principal block,
J. reine angew. Math. 546 (2002), 1--45.

\bibitem[FS2]{FS2}
R.~Farnsteiner, A.~Skowro\'nski,
The tame infinitesimal groups of odd characteristic,
Adv. Math. 205 (2006), 229--274.

\bibitem[F]{Fu}
K.~R.~Fuller,
Weakly symmetric rings of distributive module type,
Communications in Algebra 5(9) (1977), 997-1008.

\bibitem[GR]{GR}
P.~Gabriel, Ch.~Riedtmann, 
Group representations without groups,
Comment. Math. Helv. 54 (1979), 240--287.

\bibitem[GP]{GP}
I.~M.~Gelfand, V.~A.~Ponomarev,
Indecomposable representations of the Lorentz group,
Uspehi Mat. Nauk 23 (1968), 3--60.

\bibitem[J]{J}
G.~J.~Janusz,
Indecomposable modules for finite groups,
Ann. of Math. 89 (1969), 209--241.

\bibitem[K]{K}
H.~Kupisch,
Projective Moduln endlicher Gruppen mit zyklischer $p$-Sylow Gruppe,
J. Algebra 10 (1968), 1--7.

\bibitem[NN]{NN}
T.~Nakayama, C.~Nesbitt,
Note on symmetric algebras,
Ann. of Math. 39 (1938), 659--668.

\bibitem[PS]{PS}
Z.~Pogorza\l{}y, A.~Skowro\'nski,
Selfinjective biserial standard algebras,
J. Algebra 138 (1991), 491--504.

\bibitem[Rie]{Rie}
C.~Riedtmann,
Representation-finite self-injective algebras of class $\mathbbm{A}_n$,
in: Representation Theory II, Lecture Notes in Math. 832, Springer-Verlag, Berlin 1980, 449--520.

\bibitem[Ro]{Ro}
K.~W.~Roggenkamp, 
Biserial algebras and graphs,
in: Algebras and Modules II, CMS Conf. Proc. 24, Amer. Math. Soc., Providence, RI, (1998), 481--496.

\bibitem[Sch]{Sch}
S.~Schroll, 
Trivial extensions of gentle algebras and Brauer graph algebras,
J. Algebra 444 (2015), 183--200.

\bibitem[S]{S}
A.~Skowro\'nski,
Selfinjective algebras: finite and tame type,
in: Trends in Representation Theory of Algebras and Related Topics, in: Contemp. Math 406, Amer. Math. Soc., Providence, RI, (2006), 169--238.

\bibitem[SS]{SS}
D.~Simson, A.~Skowro\'nski,
Elements of the Representation Theory of Associative Algebras 3: Representation-Infinite Tilted Algebras,
London Math. Soc. Student Texts 72, Cambridge University Press, Cambridge 2007.

\bibitem[SW]{SW}
A.~Skowro\'nski, J.~Waschb\"usch,
Representation-finite biserial algebras,
J. reine angew. Math. 345 (1983), 172--181.

\bibitem[SY1]{SY1}
A.~Skowro\'nski, K.~Yamagata,
Frobenius Algebras I. Basic Representation Theory, Eur. Math. Soc. Textbooks Math., Eur. Math. Soc., Z\"urich 2011.

\bibitem[SY2]{SY2}
A.~Skowro\'nski, K.~Yamagata,
Frobenius Algebras II. Tilted and Hochschild Extensions Algebras, Eur. Math. Soc. Textbooks Math., Eur. Math. Soc., Z\"urich 2017.

\bibitem[VFCB]{CBVF}
R.~Vila-Freyer, W.~Crawley-Boevey
The structure of biserial algebras,
J. London Math. Soc. (2) 57 (1998).

\bibitem[WW]{WW}
B.~Wald, J.~Waschb\"usch,
Tame biserial algebras,
J. Algebra 95 (1985), 480--500.


\end{thebibliography}

\end{document}